%% file: Definable_sets.tex
\documentclass[12pt]{amsart}

\title[Definable sets with definable elements]{When does every definable nonempty set have a definable element?}

\author[Dorais]{\Francois\ G. Dorais}
 \address[F.~Dorais]
        {Department of Mathematics,
         University of Vermont,
          Burlington, VT 05405}
 \email{francois.dorais@uvm.edu}
 \urladdr{http://logic.dorais.org/}

\author[Hamkins]{Joel David Hamkins}
 \address[J. D. Hamkins]
        {Mathematics, Philosophy, Computer Science,
          The Graduate Center of The City University of New York,
          365 Fifth Avenue, New York, NY 10016
          \&
          Mathematics,
          College of Staten Island of CUNY,
          Staten Island, NY 10314}
\email{jhamkins@gc.cuny.edu}
\urladdr{http://jdh.hamkins.org}

\thanks{The research of the second author has been supported by grant \#69573-00 47 from the CUNY Research Foundation. This inquiry grew out of a series of questions and answers posted on MathOverflow \cite{MO10415Hamkins2010:DefinableCollectionsWithoutDefinableMembers, MO180734Hamkins2014:ConsistentEveryDefinableSetHasDefinableMember, MO180850Hamkins2014:EverySigma2DefinableSetHasODelement} and the exchange of the authors there. Commentary concerning this article can be made at http://jdh.hamkins.org/definable-sets-with-definable-elements.}

\usepackage{amssymb}
\usepackage[hidelinks]{hyperref}
\include{MathMacrosJDH}

\begin{document}

\begin{abstract}
 The assertion that every definable set has a definable element is equivalent over \ZF\ to the principle $V=\HOD$, and indeed, we prove, so is the assertion merely that every $\Pi_2$-definable set has an ordinal-definable element. Meanwhile, every model of \ZFC\ has a forcing extension satisfying $V\neq\HOD$ in which every $\Sigma_2$-definable set has an ordinal-definable element. Similar results hold for $\HOD(\R)$ and $\HOD(\Ord^\omega)$ and other natural instances of $\HOD(X)$.
\end{abstract}

\maketitle

\noindent

\section{Introduction}

It is not difficult to see that the models of \ZF\ set theory in which every definable nonempty set has a definable element are precisely the models of $V=\HOD$. Namely, if $V=\HOD$, then there is a definable well-ordering of the universe, and so the $\HOD$-least element of any definable nonempty set is definable; and conversely, if $V\neq\HOD$, then the set of minimal-rank non-\OD\ sets is definable, but can have no definable element.

In this brief article, we shall identify the limit of this elementary observation in terms of the complexity of the definitions. Specifically, we shall prove that $V=\HOD$ is equivalent to the assertion that every $\Pi_2$-definable nonempty set contains an ordinal-definable element, but that one may not replace $\Pi_2$-definability here by $\Sigma_2$-definability, in light of theorem \ref{Theorem.VnotHODbutSigma2SetsHaveODMembers}, which shows that every model of \ZFC\ has a forcing extension satisfying $V\neq\HOD$ in which every $\Sigma_2$-definable nonempty set contains an ordinal-definable element. That theorem is proved in a manner reminiscent of several proofs of the maximality principle \cite{Hamkins2003:MaximalityPrinciple}, where one undertakes a forcing iteration attempting at each stage to force and then preserve a given $\Sigma_2$ assertion.

\section{Background, and a metamathematical issue}

An object $a$ in a model $M$ is {\df definable}, if there is a formula $\varphi$ in the language of $M$ such that $a$ is the unique satisfying instance of $\varphi$ in $M$, that is, if $M\satisfies\varphi[x]$ just in case $x=a$. The object is $\Sigma_2$-definable or $\Pi_2$-definable, for example, if there is a defining formula $\varphi$ of that level of complexity, respectively. In the context of set theory, a set $a$ is {\df ordinal-definable}, if there it is the unique satisfying instance of some assertion $\varphi(a,\vec\alpha)$ in the language of set theory, using ordinal parameters $\vec\alpha$. Since there is a definable ordinal pairing function, it suffices to have at most a single ordinal parameter $\alpha$ in place of $\vec\alpha$. More generally, one may consider the concept of $X$-definability, allowing parameters from $X$.

We should like to emphasize the subtle metamathematical point that the concept of definability is an external essentially model-theoretic notion in set theory. It doesn't in general make sense, for example, to refer in set theory to ``the class of definable elements,'' or to say that ``every such-and-such kind of set has a definable element,'' for there are models $M\satisfies \ZFC$ in which the collection of definable elements of $M$ is not a class in $M$. (See further extended discussion of this issue in \cite{HamkinsLinetskyReitz2013:PointwiseDefinableModelsOfSetTheory}.)
In this sense, the assertion that ``every definable nonempty set has a definable element'' is not {\it prima facie} expressible in the language of set theory. Nevertheless, it turns out for the reasons we mentioned in the opening paragraph of this article that a model $M\satisfies\ZF$ has the (external, model-theoretic) property that every definable nonempty set has a definable element just in case it is a model of $V=\HOD$, and so in this sense the property is actually expressible in the language of set theory.

Because of the tension with Tarski's theorem on the non-definability of truth, it is actually a remarkable fact observed by \Godel\ that the class $\OD$ of ordinal-definable sets nevertheless {\it is} a class in \ZF. It follows that the class $\HOD$ of hereditarily ordinal-definable sets is also a class, and we may express the hypothesis $V=\HOD$, asserting that every set is hereditarily ordinal-definable, as a sentence in the language of set theory. The central reason for the definability of ordinal-definability, of course, is the reflection theorem, which implies that if an object $a$ is defined by the formula $\varphi(\cdot,z)$, with parameter $z$, then there is some ordinal $\theta$ such that $a,z\in V_\theta$ and $a$ is defined by $\varphi(\cdot,z)$ inside $V_\theta$. Using this, we may express $a\in\OD$ with the assertion that there is some ordinal $\theta$, such that $a$ is definable with ordinal parameters inside the structure $\<V_\theta,\in>$. This definition of $\OD$ agrees with the external model-theoretic concept of ordinal-definability in any model of \ZF.\footnote{There is a further subtle issue here with nonstandard models, for if an $\omega$-nonstandard model $M$ thinks an object $a$ is ordinal-definable in some $V_\theta^M$, then it may be thinking this because it has a nonstandard formula $\varphi$ that it thinks defines $a$ with ordinal parameters $\vec\alpha$ in that model, and furthermore, the number of ordinal parameters $\vec\alpha$ may not be actually finite but only nonstandard finite in $M$. These issues can both be overcome as follows: first, inside $M$ we may code $\vec\alpha$ with a single ordinal using what $M$ thinks is the ordinal pairing function, and thereby reduce to the case of a single ordinal parameter $\gcode{\vec\alpha}$; next, we define $a$ in $M$ as, ``the unique object satisfying the formula coded by $\gcode{\varphi}$ using the parameters coded by $\gcode{\vec\alpha}$ in the structure $V_\theta$.'' The point is that this still defines $a$, using the code $\gcode{\varphi}$ as an additional ordinal parameter (since \Godel\ codes of formulas are natural numbers and therefore also ordinals).}

The reader may find it useful to know of the characterization of
the $\Sigma_2$ properties as the {\df semi-local properties}, those
which are equivalent to an assertion of the form $\exists\theta\
V_\theta\models\psi$, where $\psi$ can have any complexity. For a proof and further discussion of this folklore result, see the second author's blog post \href{http://jdh.hamkins.org/local-properties-in-set-theory/}{``Local properties in set theory,''} \cite{Hamkins.blog2014:Local-properties-in-set-theory}. In
particular, whenever a $\Sigma_2$ property $\varphi(A)$ is true of
a set $A$, it is because there is some $V_\theta$ satisfying
something about $A$. We may therefore preserve that $\Sigma_2$
fact about $A$, while forcing over $V$, provided that we only
force up high and preserve $V_\theta$, the rank-initial-segment of
the universe up to $\theta$.

\section{Definable sets with definable members}

Let us now state and prove the basic equivalences, which identify $\Pi_2$ as the level of complexity needed for the equivalence mentioned in the introduction of the article.

\begin{theorem}\label{Theorem.EquivalencesToV=HOD}
The following are equivalent in any model $M$ of \ZF:
 \begin{enumerate}
   \item $M$ is a model of $\text{ZFC}+\text{V}=\text{HOD}$.
   \item $M$ thinks there is a definable well-ordering of the universe.
   \item Every definable nonempty set in $M$ has a definable element.
   \item Every definable nonempty set in $M$ has an ordinal-definable element.
   \item Every ordinal-definable nonempty set in $M$ has an ordinal-definable element.
   \item Every $\Pi_2$-definable nonempty set in $M$ has an ordinal-definable element.
 \end{enumerate}
\end{theorem}

\begin{proof}
All the implications, except one, are straightforward.

($1\implies 2$) The usual HOD order is a definable well-ordering
of the universe.

($2\implies 3$) Select the least element with respect to the definable
order.

($3\implies 4$) Immediate.

($4\implies 5$) If there is an \OD-set with no \OD\ member, then the \OD-least such set is definable.

($5\implies 6$) Immediate.

($6\implies 1$) This is the non-trivial implication. To prove it, it is tempting to consider the set $A$ of minimal-rank non-$\OD$ sets, as in the proof of the implication ($4\to 1$) mentioned in the opening of this article. If $V\neq\HOD$, then this is a definable nonempty set with no ordinal-definable elements. How complex is the definition of $A$? It is not difficult to see that $A$ is $\Sigma_3$-definable. One can press this a bit harder to see that $A$ is $\Sigma_2\wedge\Pi_2$-definable, characterized by the following properties: $A$ is not
empty; all elements of $A$ have the same rank; every element of
$A$ is not in OD; every set of rank less than an element of $A$ is
in OD; every set not in $A$, but of the same rank as an element of
$A$, is in OD. Each of these properties is either $\Sigma_2$ or
$\Pi_2$, making the set $A$ to be $\Sigma_2\wedge\Pi_2$-definable.
Specifically, the first two requirements are $\Sigma_2$, being
witnessed in a rank-initial segment of the universe; the third is
$\Pi_2$; the fourth and fifth are both $\Sigma_2$, since they are
true just in case there is a large $V_\theta$ which believes them
to be true. This is close to optimal, as far as defining $A$ is concerned, since it is not provably $\Sigma_2$-definable, as in any model of $V\neq\HOD$, we could perform forcing up high so as to preserve any given $\Sigma_2$ assertion, while making some element of $A$ to be coded into the \GCH\ pattern and hence ordinal definable.

So in order to prove the implication ($6\implies 1$), we shall augment $A$ with more information. Specifically, let $A$ be the set of minimal-rank non-OD sets. That is, $A$ consists of all non-OD sets of rank $\alpha$, where $\alpha$ is minimal such that there is any non-OD set of rank $\alpha$. The desired set will be $U=A\times V_\theta$, where $\theta$ is the smallest ordinal such that $A\in V_\theta$ and $V_\theta\models A$ is the set of minimal-rank non-OD sets.

The set $U$ is defined by the following property: $U$ consists of
the cartesian product $U=A\times B$ of two sets $A$ and $B$ such
that the elements of $A$ are not in \OD\ and the set $B$ has the form $B=V_\theta$ for some ordinal
$\theta$ such that $A\in V_\theta$ and $V_\theta\models ``A$ is the
set of minimal-rank OD sets and there is no $\theta'<\theta$ for
which $V_{\theta'}\models A$ is the set of minimal-rank non-OD
sets.''

This property altogether has complexity $\Pi_2$, due mainly to the
clause asserting that elements of $A$ are not in \OD. The part requiring that $U$ has the form
$A\times B$ is $\Delta_0$. The part asserting that $B$ has
the form $B=V_\theta$ for some ordinal $\theta$ has complexity
$\Pi_1$, because this is true provided $B$ is
transitive and satisfies some minimal set theory such that it
thinks it is a $V_\theta$ and such that $B$ contains all subsets
of any of its elements, so that it is using the true power set
operation. The properties asserting that $V_\theta$, that is, $B$,
satisfies certain complicated assertions has complexity
$\Delta_0$, since all quantifiers are bounded by $B$ and hence
ultimately by $U$. And finally, asserting that the elements of $A$
are not ordinal-definable has complexity $\Pi_2$, since the relation
``$x\in\text{OD}$'' has complexity $\Sigma_2$, as any instance of
ordinal-definability reflects to some $V_\theta$ and hence is
locally verifiable; thus, the assertion $\forall x\in A\
x\notin\text{OD}$ has complexity $\Pi_2$.

So altogether, the set $U=A\times V_\theta$ is $\Pi_2$-definable,
but it can have no ordinal-definable elements, since every element
of $U$ has the form $(a,b)$ for some $a\in A, b\in V_\theta$, and
if the pair $(a,b)$ were ordinal-definable, then $a$ would be
ordinal-definable, contradicting $a\in A$ and the fact that every
member of $A$ is not ordinal-definable.
\end{proof}

Note that the proof of $(6\to 1)$ is completely uniform, in that the definition
of the set $U$ does not depend on the model in any way. Rather, we
have a $\Pi_2$ definition that $\ZFC+V\neq\HOD$ proves is a
nonempty set disjoint from OD.

Let us now show that the $\Pi_2$-definability clause in statement 6 of the main theorem cannot be changed to $\Sigma_2$-definability.

\begin{theorem}\label{Theorem.VnotHODbutSigma2SetsHaveODMembers}
Every model of $\ZFC$ has a forcing extension
satisfying $V\neq\HOD$, in which every $\Sigma_2$-definable set
has a definable element.
\end{theorem}

The proof idea is that we shall perform a forcing iteration, considering each
$\Sigma_2$ formula in turn, where we try to freeze if possible the set defined
by that formula (in some suitable forcing extension) and then code one of its elements (if any) into
the GCH pattern high above the witness to that $\Sigma_2$
property. In the end, every nonempty $\Sigma_2$-definable set will
contain an ordinal-definable element and hence a definable element.

\begin{proof}
Start with $V$ as a ground model. Enumerate the
$\Sigma_2$ formulas $\varphi_0$, $\varphi_1$, $\varphi_2$, and so on. Note
that we may refer to $\Sigma_2$-truth since there is a universal
truth predicate for truth of bounded complexity (so there will be
no issues with Tarski's theorem on the non-definability of truth).
We define a full-support forcing iteration $\P$ of length
$\omega$, where the forcing at each stage will become
progressively more highly closed. At the first stage, we consider
the formula $\varphi_0$, and ask: is there a forcing extension
$V[g_0]$ in which $\varphi_0$ holds of a nonempty set $A_0$? If
so, we perform such a forcing (let the generic filter choose one amongst the set of minimal-rank such instances), and let $\lambda_0$ be the smallest $\beth$-fixed point
above the size of that forcing so that also $\varphi_0$ is
witnessed in $V_{\lambda_0}^{V[g_0]}$. Next, perform additional
$\leq\lambda_0$-closed forcing over $V[g_0]$ to an extension
$V[g_0][h_0]$, where $h_0$ forces to code one of the elements of
$A_0$ into the GCH pattern above $\lambda_0$. This preserves the
definition of $A_0$ by $\varphi_0$, while ensuring that $A_0$ has
an ordinal definable element. Now, let $\theta_1$ be well above
this coding, and continue.

At stage $n\geq 1$, we have forced to the partial extension
    $$V^{(n)}=V[g_0][h_0]\cdots[g_{n-1}][h_{n-1}]$$
by performing forcing below the cardinal $\theta_n$, which we had defined at the end of the previous stage. At this stage, we ask whether we can
perform further $\leqtheta_n$-closed forcing over this model in such a way that $\varphi_n$ will hold of a
nonempty set $A_n$ in the resulting extension. If so, we do that forcing. Let $\lambda_n$ be large enough to witness the $\Sigma_2$ property for $\varphi_n(A_n)$, and then perform \GCH\ coding forcing above this so as to
make an element of $A_n$ ordinal-definable, and let $\theta_{n+1}$ be larger than all that. If it was not possible to perform forcing so that $\varphi_n$ would hold of some $A_n$, then we ignore $\varphi_n$ and let
$\theta_{n+1}=\theta_n$.

Suppose $G\of\P$ is $V$-generic for the forcing we have described. Let $\delta$ be any regular cardinal above $\sup_n\theta_n$, and let $H\of\delta$ be a $V[G]$-generic Cohen subset of $\delta$. Consider the resulting forcing extension $V[G][H]$, our final model. Because we added the Cohen subset of $\delta$, which is ordinal-definable homogeneous forcing, it follows that $V[G][H]$ satisfies $V\neq\HOD$.

Nevertheless, we claim that every $\Sigma_2$-definable nonempty set in our model $V[G][H]$ has a definable element. Note first that because we used full support, it follows that the tail forcing in $\P$ after stage $n$
is $\leq\theta_n$-closed, as is the forcing to add $H$, and so this tail forcing
adds no new sets of rank below $\lambda_n$. Thus, if $\varphi_n$
defines a nonempty set in $V[G][H]$, then at stage $n$ we would
have observed that it was possible to force $\varphi_n$ to hold of a
nonempty set (with forcing that was sufficiently closed), and so
we would have treated it at stage $n$. That is, we would have
forced to code one of its elements into the GCH pattern,
afterwards always preserving that definition and this coding. So
in the case that $\varphi_n$ does define a nonempty set in
$V[G][H]$, then the stage $n$ forcing exactly ensured that one of
the elements of this set was coded into the GCH pattern of
$V[G][H]$ and was therefore ordinal-definable there. The later
stages of forcing were arranged so as to preserve all these
definitions. Since the set defined by $\varphi_n$ has an ordinal-definable element, the $\OD$-least such element is actually definable. So every $\Sigma_2$-definable nonempty set in $V[G][H]$ has a definable element, as desired.
\end{proof}

The proof of theorem~\ref{Theorem.VnotHODbutSigma2SetsHaveODMembers} has a certain resemblence to the the second author's forcing iteration proof of the maximality principle \cite{Hamkins2003:MaximalityPrinciple}, which one considers each sentence in turn, forcing it in such a way that it remains true in all further forcing extensions, if this is possible. The end result is a model
where any statement that could become necessarily true by forcing,
is already true, and this is precisely what the maximality
principle asserts.

\section{Allowing parameters}

Let us now generalize the previous analysis to allow parameters from an arbitrary $\Sigma_2$-definable class $X$, such as $X=\Ord$, or $X=\R$ or $X=\Ord^\omega$, corresponding to inner models $\HOD$, $\HOD(\R)$, and $\HOD(\Ord^\omega)$; the latter model $\HOD(\Ord^\omega)$ can be fruitfully viewed as an analogue of the Chang model $L(\Ord^\omega)$. We define $\OD(X)$ as the class of sets $x$ that are definable in some $V_\theta$ using parameters from $X$, which effectively also adds the ordinals as parameters, and $\HOD(X)$ is the class of sets hereditarily in $\OD(X)$. If $X$ is a proper class, then one doesn't ever actually need any ordinal parameters, since for every ordinal $\alpha$, there is an element $x\in X$ whose rank is the $\alpha^{th}$ among any elements of $X$, and so in this case, every ordinal is definable from an element of $X$. When $X$ is a set, however, such as $X=\R$, then in general we have no first-order expressible concept of $\R$-definable, and we form $\HOD(\R)$ by allowing also ordinal parameters.

Note that if $X$ is $\Sigma_2$ definable, then the class $\OD(X)$ is also $\Sigma_2$ definable, since $a\in\OD(X)$ just in case there is an ordinal $\theta$ and elements $\vec x$ from $X$ such that $a$ is definable from ordinal parameters and $\vec x$ in $V_\theta$, which also verifies that $\vec x$ are in $X$. So membership in $\OD(X)$ is verified by a property observable in some $V_\theta$, which as we mentioned earlier characterizes the $\Sigma_2$ properties.

Indeed, there is a $\Delta_2$-definable surjection of $X^{<\omega}\times\Ord$ onto $\OD(X)$. Namely, every element $a\in\OD(X)$ is definable in some $V_\theta$ by ordinal parameters and elements of $X$, and so we can map $(\vec x,\alpha)\mapsto a$, where $\alpha$ is an ordinal coding $\theta$ and the ordinal parameters below $\theta$ and the \Godel\ code of the formula being used. This map is $\Delta_2$ definable, since in any sufficiently large $V_\beta$ we can correctly recognize whether or not $(\vec x,\alpha)\mapsto a$.

\begin{theorem}
The following are equivalent in any model $M$ of \ZF, with any $\Sigma_2$-definable class $X$:
 \begin{enumerate}
   \item $M$ is a model of $\text{ZFC}+\text{V}=\HOD(X)$.
   \item $M$ has a $\Delta_2$-definable surjection of $X^{<\omega}\times\Ord$ onto $M$.
   \item Every definable nonempty set in $M$ has an $X$-definable element.
   \item Every definable nonempty set in $M$ has an $(\Ord,X)$-definable element, that is, an element in $\OD(X)$.
   \item $M$ thinks that every nonempty set in $\OD(X)$ has an element in $\OD(X)$.
   \item Every $\Pi_2$-definable nonempty set in $M$ has an member in $\OD(X)$.
 \end{enumerate}
\end{theorem}

\begin{proof}
($1\to 2$) If $M$ is a model of $V=\HOD(X)$, then the $\Delta_2$-definable surjection of $X^{<\omega}\times\Ord$ onto $\OD(X)$ is actually onto $M$.

($2\to 3$) If there is a definable surjection $d:X^{<\omega}\times\Ord\to M$ and $A$ is a definable nonempty set, then there is some $a=d(\vec x,\alpha)\in A$ for some $\vec x\in X^{<\omega}$. Let $\beta\leq\alpha$ be least such that $d(\vec x,\beta)\in A$. The object $d(\vec x,\beta)$ is in $A$ and definable from parameters $\vec x$.

($3\to 4$) Immediate.

($4\to 5$) If there is a nonempty set in $\OD(X)$ with no elements in $\OD(X)$, then let $A$ be the union of all minimal-rank such sets. Since $X$ is definable, the set $A$ is a definable set with no members in $\OD(X)$.

($5\to 6$) Immediate.

($6\to 1$) Assume $V\neq\HOD(X)$. Since $X$ itself is $\Sigma_2$-definable, it follows that if some $V_\theta$ thinks $x\in X$, then it is right about that. Let $A$ be the set of minimal-rank non-$\OD(X)$ sets, and let $\theta$ be least such that $V_\theta$ can see that all the other members of that rank or of smaller rank are in $\OD(X)$. The set $A\times V_\theta$ is now $\Pi_2$-definable, since we need only say that the members of $A$ are not in $\OD(X)$ and $V_\theta$ thinks that all the other members of that rank or less are in $\OD(X)$ and that no smaller $\theta'<\theta$ thinks that. But no member of $A\times V_\theta$ can be in $\OD(X)$, contrary to (6).
\end{proof}

Notice that the generalization of theorem \ref{Theorem.VnotHODbutSigma2SetsHaveODMembers} to the context with parameters is a consequence of theorem \ref{Theorem.VnotHODbutSigma2SetsHaveODMembers} itself.

\begin{theorem}
Suppose $X$ is a $\Sigma_2$-definable class with the property that $V\neq\OD(X)$ is forceable over any forcing extension, by forcing preserving any desired $V_\theta$. Then there is a forcing extension of the universe in which $V\neq\OD(X)$, yet every $\Sigma_2$-definable nonempty set has an $(\Ord,X)$-definable member.
\end{theorem}

To be clear, in the theorem we reinterpret $X$ in the forcing extensions using its $\Sigma_2$-definition; the forcing may add new elements to this definable class. In important cases such as when $X=\R$, when $X$ is a set, when $X=\Ord$, when $X=\Ord^\omega$, our forcing will not actually add new elements to $X$. The hypothesis that $V\neq\OD(X)$ is forceable is true in the cases of $X$ just mentioned, since one can add a Cohen subset to a large regular cardinal; this will not add elements to $X$ and by homogeneity one will achieve $V\neq\OD(X)$.

\begin{proof}
We use essentially the same model $V[G][H]$ provided by theorem \ref{Theorem.VnotHODbutSigma2SetsHaveODMembers}, except at the top we arrange the $H$ forcing so as to ensure $V\neq\OD(X)$. The argument of theorem \ref{Theorem.VnotHODbutSigma2SetsHaveODMembers} shows that every $\Sigma_2$-definable set in $V[G][H]$ has a definable member, which is therefore also $(\Ord,X)$-definable.
\end{proof}

\bibliographystyle{alpha}
\bibliography{MathBiblio,HamkinsBiblio,WebPosts}

\end{document}

%% file: MathMacrosJDH.tex
\newtheorem{theorem}{Theorem}

\newtheorem*{maintheorem*}{Main Theorem}
\newtheorem*{maintheorems*}{Main Theorems}

\newtheorem*{corollary*}{Corollary}
\newtheorem*{corollaries*}{Corollaries}

\newtheorem*{question*}{Question}

\newtheorem*{questions*}{Questions}
\newtheorem*{mainquestion*}{Main Question} 
\newtheorem*{openquestion*}{Open Question} 

\newcommand{\QED}{\end{proof}}

\def\proclaim[#1]{{\bf #1}}
\def\BF#1.{{\bf #1.}}

\newcommand{\Francois}{Fran\c{c}ois}
\newcommand{\Godel}{G\"odel}

\renewcommand{\P}{{\mathbb P}}

\newcommand{\R}{{\mathbb R}}

\newcommand{\of}{\subseteq}

\newcommand{\satisfies}{\models}

\newcommand{\smalllt}{\mathrel{\mathchoice{\raise2pt\hbox{$\scriptstyle<$}}{\raise1pt\hbox{$\scriptstyle<$}}{\raise0pt\hbox{$\scriptscriptstyle<$}}{\scriptscriptstyle<}}}
\newcommand{\smallleq}{\mathrel{\mathchoice{\raise2pt\hbox{$\scriptstyle\leq$}}{\raise1pt\hbox{$\scriptstyle\leq$}}{\raise1pt\hbox{$\scriptscriptstyle\leq$}}{\scriptscriptstyle\leq}}}

\newcommand{\leqtheta}{{{\smallleq}\theta}}

\newcommand{\boolval}[1]{\mathopen{\lbrack\!\lbrack}\,#1\,\mathclose{\rbrack\!\rbrack}}
\def\[#1]{\boolval{#1}}
\newbox\gnBoxA
\newdimen\gnCornerHgt
\setbox\gnBoxA=\hbox{\tiny$\ulcorner$}
\global\gnCornerHgt=\ht\gnBoxA
\newdimen\gnArgHgt
\def\gcode #1{%
\setbox\gnBoxA=\hbox{$#1$}%
\gnArgHgt=\ht\gnBoxA%
\ifnum     \gnArgHgt<\gnCornerHgt \gnArgHgt=0pt%
\else \advance \gnArgHgt by -\gnCornerHgt%
\fi \raise\gnArgHgt\hbox{\tiny$\ulcorner$} \box\gnBoxA %
\raise\gnArgHgt\hbox{\tiny$\urcorner$}}
\newcommand{\UnderTilde}[1]{{\setbox1=\hbox{$#1$}\baselineskip=0pt\vtop{\hbox{$#1$}\hbox to\wd1{\hfil$\sim$\hfil}}}{}}
\newcommand{\Undertilde}[1]{{\setbox1=\hbox{$#1$}\baselineskip=0pt\vtop{\hbox{$#1$}\hbox to\wd1{\hfil$\scriptstyle\sim$\hfil}}}{}}
\newcommand{\undertilde}[1]{{\setbox1=\hbox{$#1$}\baselineskip=0pt\vtop{\hbox{$#1$}\hbox to\wd1{\hfil$\scriptscriptstyle\sim$\hfil}}}{}}
\newcommand{\UnderdTilde}[1]{{\setbox1=\hbox{$#1$}\baselineskip=0pt\vtop{\hbox{$#1$}\hbox to\wd1{\hfil$\approx$\hfil}}}{}}
\newcommand{\Underdtilde}[1]{{\setbox1=\hbox{$#1$}\baselineskip=0pt\vtop{\hbox{$#1$}\hbox to\wd1{\hfil\scriptsize$\approx$\hfil}}}{}}

\renewcommand{\implies}{\mathrel{\rightarrow}}

\def\<#1>{\left\langle#1\right\rangle}

\newcommand{\Ord}{\mathord{{\rm Ord}}}

\newcommand{\ZFC}{{\rm ZFC}}
\newcommand{\ZF}{{\rm ZF}}

\newcommand{\GCH}{{\rm GCH}}

\newcommand{\HOD}{{\rm HOD}}
\newcommand{\OD}{{\rm OD}}

\newcommand{\cell}[1]{\boxit{\hbox to 17pt{\strut\hfil$#1$\hfil}}}
\newcommand{\head}[2]{\lower2pt\vbox{\hbox{\strut\footnotesize\it\hskip3pt#2}\boxit{\cell#1}}}
\newcommand{\boxit}[1]{\setbox4=\hbox{\kern2pt#1\kern2pt}\hbox{\vrule\vbox{\hrule\kern2pt\box4\kern2pt\hrule}\vrule}}
\newcommand{\Col}[3]{\hbox{\vbox{\baselineskip=0pt\parskip=0pt\cell#1\cell#2\cell#3}}}
\newcommand{\tapenames}{\raise 5pt\vbox to .7in{\hbox to .8in{\it\hfill input: \strut}\vfill\hbox to
.8in{\it\hfill scratch: \strut}\vfill\hbox to .8in{\it\hfill output: \strut}}}
\newcommand{\Head}[4]{\lower2pt\vbox{\hbox to25pt{\strut\footnotesize\it\hfill#4\hfill}\boxit{\Col#1#2#3}}}
\newcommand{\Dots}{\raise 5pt\vbox to .7in{\hbox{\ $\cdots$\strut}\vfill\hbox{\ $\cdots$\strut}\vfill\hbox{\
$\cdots$\strut}}}
\newcommand{\df}{\it} 